\documentclass[11pt]{amsart}

\usepackage[T1]{fontenc}
\usepackage{lmodern}
\usepackage{microtype}
\usepackage{amsmath,amssymb,mathtools}
\usepackage{enumitem}
\usepackage{needspace}
\usepackage[dvipsnames]{xcolor}
\usepackage{graphicx}
\usepackage{caption}
\usepackage{tikz}
\usetikzlibrary{arrows.meta}
\usepackage{placeins}
\usepackage{hyperref}

\hypersetup{
  colorlinks=true,
  linkcolor=MidnightBlue,
  citecolor=MidnightBlue,
  urlcolor=BrickRed,
  pdftitle={Reachability under Arc Crossing Changes},
  pdfauthor={Bo Chen, Jinbo Geng and Zerui Wu}
}

\newtheorem{theorem}{Theorem}[section]
\newtheorem{lemma}[theorem]{Lemma}
\newtheorem{proposition}[theorem]{Proposition}
\newtheorem{corollary}[theorem]{Corollary}
\theoremstyle{definition}
\newtheorem{definition}[theorem]{Definition}

\newtheorem{example}[theorem]{Example}
\newtheorem*{definitionx}{Definition}
\theoremstyle{remark}
\newtheorem{remark}[theorem]{Remark}

\newcommand{\F}{\mathbb F_2}

\newcommand{\uAG}{\underline{A}_G}
\newcommand{\one}{\mathbf 1}

\tikzset{
  knotcurve/.style={
    draw=black,line width=.8pt,line cap=round,line join=round
  },
  knotbridge/.style={
    knotcurve,preaction={draw=white,line width=4.8pt}
  },
  selectedarc/.style={
    draw=ForestGreen!75!black,line width=4.8pt,
    opacity=.38,line cap=round,line join=round
  },
  crossingring/.style={
    circle,draw=ForestGreen!75!black,line width=1pt,
    minimum size=5.5mm,inner sep=0pt
  },
  changedring/.style={
    circle,draw=BrickRed,line width=1pt,
    minimum size=5.5mm,inner sep=0pt
  },
  crossing/.style={
    circle,fill=black,inner sep=0pt,minimum size=2.2mm
  }
}

\newcommand{\FigureEightSegment}[2]{%
  \draw[#1]
    plot[samples=180,smooth,domain=#2,variable=\t]
    ({(2+cos(2*\t))*cos(3*\t)},
     {(2+cos(2*\t))*sin(3*\t)});%
}

\setlist[enumerate]{topsep=3pt,itemsep=2pt}

\title[Reachability under arc crossing changes]
{Reachability under Arc Crossing Changes}

\author{Bo Chen, Jinbo Geng and Zerui Wu}

\address{School of Mathematics and Statistics, Huazhong University of Science and Technology, Wuhan, China}
\email{bobchen@hust.edu.cn}

\address{School of Mathematical Science, Zhejiang Normal University, Jinhua, China}
\email{jinbogeng@zjnu.cn}
\email{wuzerui2026@163.com}

\subjclass[2020]{Primary 57K10; Secondary 05C40, 57M15}
\keywords{knot shadow, arc crossing change, strong reachability, source state,
parity, arc-crossing state digraph, alternating diagram}

\begin{document}

\begin{abstract}
Cericola proved that every knot diagram can be transformed into an ascending
unknot diagram by arc crossing changes.  We refine this result for all
diagrams on a fixed R1-reduced classical knot shadow with more than three
crossings.  Apart from at most two
source  diagrams, any two diagrams of the same state parity
are mutually reachable.  Each source diagram reaches every
non-source diagram of its parity but cannot be reached from another
diagram.  A local five-occurrence condition on the marked Gauss word
characterizes the sources and thereby determines the full directed
reachability relation.
\end{abstract}

\maketitle

\section{Introduction}\label{sec:introduction}

Two operations in the literature are called \emph{arc crossing change}.
Kinuno's move acts on a semi-arc, whose position is fixed by the
shadow~\cite{Kinuno2022}.  Cericola's move acts on an arc between consecutive
undercrossings~\cite[Section~2.1]{Cericola2024}.   Throughout this
paper, arc crossing change means Cericola's version.

Cericola's state-dependent arcs produce a directed system. After a move the
arc partition changes, so the move need not be immediately reversible.  This
behavior does not occur for Kinuno's fixed semi-arcs.  The two operations are
indeed inequivalent in the figure-eight example of Cheng, Liao, and Song
\cite[Example~4.1]{ChengLiaoSong2025}.

Cericola proved that every knot diagram reaches an ascending unknot diagram
by a sequence of arc crossing changes~\cite[Theorem~3.1]{Cericola2024}.
That theorem guarantees an unknotting path, but it does not decide whether a
prescribed diagram is reachable, or when a change can be reversed.  We answer
these questions.

Let \(G\) be a classical knot shadow with \(s\) crossings.  A knot diagram on
\(G\) is obtained by choosing the undercrossing occurrence at every crossing.
After ordering the two occurrences of each crossing, define the \emph{parity} of the diagram as the parity of the number of second occurrences
chosen.  Every arc crossing change switches two choices, so this parity is
preserved.  Call a diagram a \(\emph{source}\) if no other diagram on \(G\)
can be changed into it by one arc crossing change.

\begin{theorem}[Reachability of knot diagrams]\label{thm:intro-diagrams}
Let \(G\) be a classical knot shadow with no Reidemeister-I curl and with
\(s>3\) crossings.  A diagram \(D\) on \(G\) can be transformed into another
diagram \(D'\) on \(G\) by a sequence of arc crossing changes if and only if
\begin{enumerate}
  \item \(D\) and \(D'\) have the same state parity, and
  \item \(D'\) is not a source.
\end{enumerate}

\end{theorem}

The source diagrams admit a local description.  Mark each occurrence
in the cyclic Gauss word by \(\mathrm U\) or \(\mathrm O\), according as it
is an undercrossing or overcrossing occurrence.

\begin{theorem}[Source diagrams]\label{thm:intro-source}
A diagram on an R1-reduced classical knot shadow with more than three
crossings is a source if and only if:
\begin{enumerate}
  \item the diagram is alternating; and
  \item its marked Gauss word has no cyclic five-occurrence block
  \[
    \mathrm O_i\,\mathrm U_j\,\mathrm O_k\,\mathrm U_i\,\mathrm O_j
  \]
  with \(i,j,k\) pairwise distinct.
\end{enumerate}
In particular, there are at most two source diagrams.
\end{theorem}

Cheng, Liao, and Song encode the diagrams on \(G\) as the vertices of a
directed graph \(A_G\), with arrows given by single arc crossing changes
\cite[Section~4.1]{ChengLiaoSong2025}.  Theorem~\ref{thm:intro-diagrams}
says that, in either parity class, all non-source diagrams form one
strongly connected component, while each source diagram is a singleton
source component.   This
determines all ordered reachability relations on the fixed shadow, refining
Cericola's existence of an unknotting path.

The proof first translates source diagrams into the local word condition
of Theorem~\ref{thm:intro-source}, then connects all remaining states through
token digraphs (see Section~\ref{sec:non-source-scc}). 

Section~\ref{sec:states} fixes the notation. Section~\ref{sec:reverse} proves
the two theorems in terms of the state graph. 

\section{Preliminaries and notation}\label{sec:states}

Basic knot-diagram terminology follows Lickorish~\cite[Chapter~1]
{Lickorish1997}.  The description below uses the standard parameter-circle
viewpoint of Gauss diagrams~\cite[Section~1]{PolyakViro1994} and cyclic Gauss
words~\cite[Introduction and Section~5.1]{Turaev2004}.

\subsection{Unicursal shadows and occurrence words}

\begin{definition}[Shadow and occurrence word]\label{def:shadow}
A \emph{classical knot shadow} is the image \(G\subset\mathbb R^2\) of a
generic immersion \(\gamma:S^1\to\mathbb R^2\) with only transverse double
points, called \emph{crossings}.  Enumerate them as
\(x_0,\ldots,x_{s-1}\), and use the integer \(i\) as the label of \(x_i\).
Each crossing has two preimages
\[
  \gamma^{-1}(x_i)=\{p_i^0,p_i^1\}\subset S^1,
\]
called its two \emph{occurrences}.  Orient \(S^1\) and traverse it from a
basepoint outside these \(2s\) occurrences.  Label the two preimages so that
\(p_i^0\) is met first and \(p_i^1\) second; these are the \emph{first} and
\emph{second occurrences} of \(i\).  If \(q_r\), \(0\le r<2s\), is the
\(r\)-th preimage met, then \(r\) is its \emph{occurrence position}.
Recording the label of
\(\gamma(q_r)\) gives
\begin{equation}\label{eq:word}
  W=w_0w_1\cdots w_{2s-1},  w_j\in \{0,1,\cdots, s-1\}.
\end{equation}
Every label occurs twice.  We regard \(W\) cyclically, so changing the
basepoint rotates it; changing the labels relabels its letters.  
Position indices are read modulo \(2s\).
\end{definition}

The domain is one circle, so the shadow is unicursal. Link shadows are not
considered.  We assume \(s>3\). The low-crossing cases are recorded in
Remark~\ref{rem:low-crossings}.

\begin{definition}[Semi-arc]\label{def:semiarc}
The \emph{semi-arc} \(h_r\) is the image under \(\gamma\) of the oriented
interval in \(S^1\) from \(q_r\) to \(q_{r+1}\). 
\end{definition}

The \(2s\) semi-arcs depend only on the shadow.   If two cyclically
adjacent letters of \(W\) agree, the intervening semi-arc bounds an R1
curl, and we call \(G\), or \(W\), \emph{R1-reduced} when no such pair
occurs.

\subsection{States, arcs, and moves}

For each crossing \(x_i\), choose one point of
\(\gamma^{-1}(x_i)\) as its undercrossing preimage; its mate is then the
overcrossing preimage.  We call such a choice a \emph{state} \(\alpha\) and
write
\[
  \operatorname{Sel}(\alpha)
  =\{\text{undercrossing preimages selected by }\alpha\}\subset S^1.
\]
For a state \(\alpha\), define
\[
  \alpha_i=
  \begin{cases}
    0,&\text{if \(p_i^0\) is selected},\\
    1,&\text{if \(p_i^1\) is selected}.
  \end{cases}
\]
Thus \(\alpha\mapsto(\alpha_0,\ldots,\alpha_{s-1})\) is a bijection from
the state set to \(\F^s\), where \(\F=\{0,1\}\) is the \(2\)-element
field and each coordinate records which of the two occurrences is the
undercrossing.  Write \(e_i\) for the \(i\)-th standard basis vector and
\(\one=(1,\ldots,1)\).

\begin{definition}[Marked occurrence word]\label{def:marked-word}
Mark each occurrence in \(W\) by \(\mathrm U\) when it belongs to
\(\operatorname{Sel}(\alpha)\), and by \(\mathrm O\) otherwise.  The
resulting \(\mathrm O/\mathrm U\)-Gauss code is denoted by \(W_\alpha\) and called the
\emph{marked occurrence word}.  The state is \emph{alternating} when its
marks alternate cyclically.
\end{definition}

The \(\mathrm O/\mathrm U\)-Gauss code is standard
\cite[Section~3.2]{Kauffman1999}. Only the term \emph{marked occurrence word}
and the notation \(W_\alpha\) are specific to this paper.

\begin{definition}[Arc]\label{def:arc}
An \emph{arc} of \(\alpha\) is the oriented segment from one point of
\(\operatorname{Sel}(\alpha) \subseteq S^1\) to the next under the cyclical order induced by the orientation of $S^1$.  Thus the selected set cuts the
circuit into \(s\) arcs.
\end{definition}

\begin{definition}[Arc crossing change]\label{def:arc-change}
Let \(c\) be an arc of $\alpha$ with endpoint crossings \(i\) and \(j\).  An \emph{arc
crossing change} on \(c\) switches the over-under information at these two
crossings and leaves every other crossing unchanged.  Equivalently, the two
endpoint occurrences of \(c\) become overcrossings and their mates become
undercrossings.
\end{definition}

Definitions~\ref{def:arc-change} are due to
Cericola~\cite[Section~2.1]{Cericola2024}.  In state coordinates the change is
\begin{equation}\label{eq:move}
  \alpha'=\alpha+e_i+e_j.
\end{equation}
\begin{definition}[Arc-crossing state graph]\label{def:AG}
Following Cheng, Liao, and Song~\cite[Section~4.1]{ChengLiaoSong2025}, the
\emph{arc-crossing state graph} \(A_G\) has vertex set \(\F^s\).  For each
arc \(c\) of a state \(\alpha\) with endpoint crossings \(i,j\), the arc
crossing change on \(c\) gives the arrow
$
  \alpha\longrightarrow\alpha+e_i+e_j,
$
 and \(A_G\) has exactly these arrows.
\end{definition}

\begin{remark}\label{rem:checkerboard}
Cheng, Liao, and Song encode states by a checkerboard-coloring convention
rather than the first/second-occurrence convention used here.  The two encodings differ by an affine
change of binary coordinates, so Proposition~\ref{prop:encoding}
identifies the resulting directed graphs.
\end{remark}

\begin{proposition}\label{prop:encoding}
The directed-graph isomorphism type of \(A_G\) is unchanged by exchanging the
first/second convention at any crossings, relabeling crossings, cyclically
rotating the word, or reversing its orientation.
\end{proposition}

\begin{proof}
Exchanging the coordinate convention at the crossings in a set
\(S\subseteq\{0,\ldots,s-1\}\) translates every state by
\(\sum_{i\in S}e_i\).  Relabeling the crossings permutes the coordinates.
Rotation changes only the displayed starting position of the cyclic
traversal.  Reversal changes the traversal orientation.  In every case there
is a bijection between states that preserves which geometric crossing
occurrences are under.  It therefore preserves the current arcs and the
result of arrows.  The induced bijection is a
directed-graph isomorphism.
\end{proof}

Each state has \(s\) distinct successors and outdegree \(s\)
\cite[Proposition~4.2(1)]{ChengLiaoSong2025}.  A move need not be reversible.
We use standard digraph terminology~\cite[Chapter~1]
{BangJensenGutin2009}: \(\uAG\) denotes the underlying undirected graph,
an SCC is a strongly connected component, a source has indegree zero, and
\[
  \operatorname{SC}(A_G)
\]
is the directed acyclic graph (called \emph{condensation digraph} of \(A_G\)) obtained by contracting its SCCs and deleting any parallel arrows obtained in this process.  Weak connectivity refers to
connectivity in \(\uAG\).

Figure~\ref{fig:occurrence-encoding} shows one transition for the
figure-eight word \(W=01231032\).

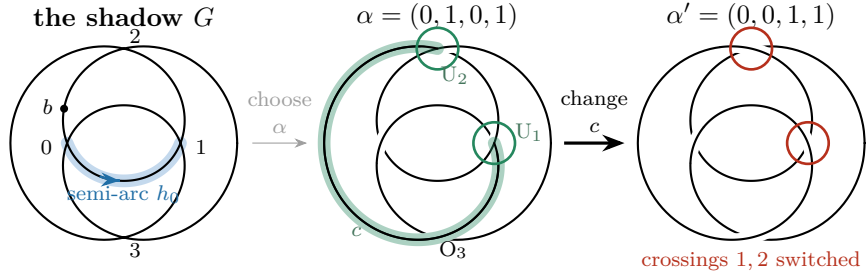
\begin{figure}[ht]
\centering
\begin{tikzpicture}[x=.5cm,y=.5cm,>=Stealth]
  \begin{scope}[shift={(-8.3,0)}]
    \node[font=\small\bfseries] at (0,3.35) {the shadow \(G\)};
    \FigureEightSegment{draw=RoyalBlue!65,line width=4.8pt,
      opacity=.35,line cap=round}{60:120}
    \FigureEightSegment{knotcurve}{0:360}
    \FigureEightSegment{draw=RoyalBlue!80!black,line width=.9pt,
      -{Stealth}}{78:88}
    \fill[black] (-1.58,.91) circle (1.5pt);
    \node[font=\scriptsize,anchor=south east] at (-1.62, 0.5) {\(b\)};
    \node[font=\scriptsize] at (-2.05,-0.1) {\(0\)};
    \node[font=\scriptsize] at (2.05,.-0.1) {\(1\)};
    \node[font=\scriptsize] at (.3,2.82) {\(2\)};
    \node[font=\scriptsize] at (.3,-2.82) {\(3\)};
    \node[font=\scriptsize,text=RoyalBlue!80!black] at (0,-1.35)
      {semi-arc \(h_0\)};
  \end{scope}

  \draw[gray!75,-{Stealth}] (-4.9,0)--(-3.45,0)
    node[midway,above,font=\scriptsize,align=center] {choose\\\(\alpha\)};

  \begin{scope}
    \node[font=\small\bfseries] at (0,3.35)
      {\(\alpha=(0,1,0,1)\)};
    \FigureEightSegment{knotcurve}{0:360}
    \FigureEightSegment{knotbridge}{24:36}
    \FigureEightSegment{knotbridge}{114:126}
    \FigureEightSegment{knotbridge}{204:216}
    \FigureEightSegment{knotbridge}{294:306}
    \FigureEightSegment{selectedarc}{150:240}
    \FigureEightSegment{knotcurve}{150:240}
    \node[crossingring] at (0,2.5) {};
    \node[crossingring] at (1.5,0) {};
    \node[font=\scriptsize,text=ForestGreen!55!black] at (.45,1.8)
      {\(\mathrm U_2\)};
    \node[font=\scriptsize,text=ForestGreen!55!black,anchor=west]
      at (1.78,.3) {\(\mathrm U_1\)};
    \node[font=\scriptsize,text=ForestGreen!55!black] at (-2.15,-2.35)
      {\(c\)};
    \node[font=\scriptsize] at (.42,-2.82) {\(\mathrm O_3\)};
  \end{scope}

  \draw[very thick,-{Stealth}] (3.35,0)--(4.95,0)
    node[midway,above,font=\scriptsize,align=center] {change\\\(c\)};

  \begin{scope}[shift={(8.3,0)}]
    \node[font=\small\bfseries] at (0,3.35)
      {\(\alpha'=(0,0,1,1)\)};
    \FigureEightSegment{knotcurve}{0:360}
    \FigureEightSegment{knotbridge}{144:156}
    \FigureEightSegment{knotbridge}{204:216}
    \FigureEightSegment{knotbridge}{234:246}
    \FigureEightSegment{knotbridge}{294:306}
    \node[changedring] at (0,2.5) {};
    \node[changedring] at (1.5,0) {};
    \node[font=\scriptsize,text=BrickRed] at (0,-3.12)
      {crossings \(1,2\) switched};
  \end{scope}
\end{tikzpicture}
\caption{The encoding \(W=01231032\) and the arc change
\(\alpha=(0,1,0,1)\to\alpha'=(0,0,1,1)\) at crossings \(1\) and \(2\).}
\label{fig:occurrence-encoding}
\end{figure}

\section{Reachability}
\label{sec:reverse}

This section proves Theorem~\ref{thm:intro-diagrams} and
Theorem~\ref{thm:intro-source}.  The argument is entirely in terms of the
occurrence word and its selected positions.  The proof of
Theorem~\ref{thm:intro-source} occupies Subsection~\ref{sec:source-states}
and identifies the source diagrams as the source states of a directed
graph.  

For state \(\alpha=(\alpha_0,\ldots,\alpha_{s-1})\), define its
\emph{parity} by
\begin{equation}\label{eq:parity}
  p(\alpha)=\alpha_0+\cdots+\alpha_{s-1}\in\F,
\end{equation}
and set
\[
  V_\varepsilon=\{\alpha\in\F^s:p(\alpha)=\varepsilon\},
  \qquad \varepsilon\in\F.
\]
Every arrow changes two coordinates, so \(V_0\) and \(V_1\) are invariant
under arc crossing changes and there is no arrows connecting them~\cite[Proposition~4.2(3)]{ChengLiaoSong2025}.

For \(\varepsilon\in\F\), set
\[
  S_\varepsilon=\{\alpha\in V_\varepsilon:\deg^-(\alpha)=0\},
  \qquad N_\varepsilon=V_\varepsilon\setminus S_\varepsilon,
\]
where $\deg^-(\alpha) $ is the indegree of $\alpha$ in $A_G$.

We now state the state-graph form of Theorem~\ref{thm:intro-diagrams}.  

\begin{theorem}[SCC decomposition]\label{thm:strong}
Let \(W\) be the R1-reduced occurrence word of a classical shadow \(G\) with
\(s>3\) crossings.  The SCCs contained in \(V_\varepsilon\) are
\(N_\varepsilon\) and the singletons \(\{\alpha\}\),
\(\alpha\in S_\varepsilon\).  In $\operatorname{SC}(A_G)$,
\(N_\varepsilon\) is the unique sink in its parity class, and every
\(\{\alpha\}\) is a source whose only  arrow enters
\(N_\varepsilon\).  Altogether there are at most two singleton source SCCs.
\end{theorem}

We prove the theorem by first recognizing the source states locally and then
connecting the remaining states through token digraphs.

\subsection{One-step predecessors}

Fix a state \(\alpha\).  At crossing \(i\), denote the occurrence in
\(\operatorname{Sel}(\alpha)\) by \(u_i\) and its mate by \(o_i\).
Two points in a finite subset of the oriented circle are \emph{cyclically
consecutive} if they are adjacent when that subset is listed in cyclic order.

For two distinct labels \(i\) and \(j\), the only state that could reach \(\alpha\)
by changing those crossings is \(\alpha^{ij}=\alpha+e_i+e_j\), with
\begin{equation}\label{eq:candidate-set}
  \operatorname{Sel}(\alpha^{ij})
  =\bigl(\operatorname{Sel}(\alpha)\setminus\{u_i,u_j\}\bigr)
   \cup\{o_i,o_j\}.
\end{equation}

\Needspace{6\baselineskip}
\begin{lemma}[Predecessor criterion]\label{lem:predecessor}
For distinct labels \(i,j\), there is an arrow
\[
  \alpha^{ij}\longrightarrow\alpha
\]
if and only if \(o_i,o_j\) are cyclically consecutive in
\(\operatorname{Sel}(\alpha^{ij})\).
\end{lemma}

\begin{proof}
The two occurrences are consecutive selections precisely when they bound an
arc of \(\alpha^{ij}\). Changing that arc gives \(\alpha\).
\end{proof}

For \(i<j\), let \(m_{ij}(\alpha)=1\) when \(o_i,o_j\) are cyclically
consecutive in \(\operatorname{Sel}(\alpha^{ij})\), and set
\(m_{ij}(\alpha)=0\) otherwise.
Lemma~\ref{lem:predecessor} gives the exact formula
\begin{equation}\label{eq:indegree}
  \deg^-(\alpha)=\sum_{0\le i<j<s}m_{ij}(\alpha).
\end{equation}

\begin{remark}[Correction to the earlier indegree formula]
Cheng, Liao, and Song state that the indegree is
\(\sum_a\binom{n_a}{2}\), where \(n_a\) is the number of 
overcrossings on an arc \(a\) of $\alpha$
\cite[Proposition~4.2(2)]{ChengLiaoSong2025}.  This counts only pairs
\(o_i,o_j\) lying on the same  arc in $\alpha$.  Lemma~\ref{lem:predecessor} shows
that a predecessor can also arise when deleting \(u_i,u_j\) makes
\(o_i,o_j\) consecutive across an arc boundary.  Hence the stated formula
does not give the full indegree in general.  For the figure-eight word
\(W=01231032\) and the state
\(\alpha=(1,0,0,0)\), its four target arcs contain \(0,0,1,3\)
overcrossing occurrences respectively.  The stated expression is therefore \(3\), whereas
Lemma~\ref{lem:predecessor} gives four predecessors ($\alpha^{01}\to \alpha$ is allowed) and hence
\(\deg^-(\alpha)=4\).
\end{remark}

\begin{center}
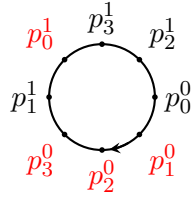

\begin{tikzpicture}[scale=.7]
  \draw[thick] (0,0) circle (1);
  \draw[->, thick, >=stealth] (-60:1) arc (-60:-82:1);
  \foreach \a in {0,-45,...,-315} { \fill (\a:1) circle (1.5pt); }
  \node[anchor=west]       at (0:1)    {$p_0^0$};
  \node[anchor=north west] at (-45:1)  {\textcolor{red}{$p_1^0$}};
  \node[anchor=north]     at (-90:1)  {\textcolor{red}{$p_2^0$}};
  \node[anchor=north east] at (-135:1) {\textcolor{red}{$p_3^0$}};
  \node[anchor=east]      at (-180:1) {$p_1^1$};
  \node[anchor=south east] at (-225:1) {\textcolor{red}{$p_0^1$}};
  \node[anchor=south]     at (-270:1) {$p_3^1$};
  \node[anchor=south west] at (-315:1) {$p_2^1$};
\end{tikzpicture}
\captionof{figure}{The marked points are the preimages of the crossings;
the red points are the points of \(\operatorname{Sel}(\alpha)\).}
\end{center}

\subsection{Source states}\label{sec:source-states}

The predecessor criterion reduces indegree zero to two local conditions.

\begin{lemma}[Every source is alternating]\label{lem:source-alternating}
If a state \(\alpha\) has indegree zero, then \(\alpha\) is alternating.
\end{lemma}

\begin{proof}
There are equally many \(\mathrm U\)- and \(\mathrm O\)-marks.  Hence a
nonalternating marking has two cyclically adjacent occurrences
\(\mathrm O_i,\mathrm O_j\), with \(i\ne j\).  They become selected and
remain consecutive in \(\alpha^{ij}\), so
Lemma~\ref{lem:predecessor} gives \(\alpha^{ij}\to\alpha\).
\end{proof}

We now isolate the only possible predecessor of an alternating state on an
R1-reduced word.

\begin{lemma}[Alternating predecessor pattern]
\label{lem:alternating-pattern}
Let \(W\) be an R1-reduced occurrence word of a shadow and let \(\alpha\) be
alternating.  Then \(\deg^-(\alpha)>0\) if and only if \(W_\alpha\) contains
five consecutive occurrences
\begin{equation}\label{eq:five-occurrence}
  \mathrm O_i\,\mathrm U_j\,\mathrm O_k\,
  \mathrm U_i\,\mathrm O_j,
\end{equation}
where \(i,j,k\) are pairwise distinct.  Moreover, every incoming arrow to
\(\alpha\) is reversible.
\end{lemma}

\begin{proof}
The displayed block makes \(o_i,o_j\) consecutive after \(u_i,u_j\) are
replaced by their mates, so \(\alpha^{ij}\to\alpha\).  Conversely, suppose
\(\alpha^{ij}\to\alpha\).  By Lemma~\ref{lem:predecessor}, one open interval
between \(o_i\) and \(o_j\) contains no selection other than possibly
\(u_i,u_j\).  Alternation leaves two cases.

If the interval contains one \(\mathrm U\)-mark, its block is
\(\mathrm O_i\mathrm U_x\mathrm O_j\), with \(x=i\) or \(j\); this gives
adjacent equal letters.  If it contains two, its block is
\[
  \mathrm O_i\,\mathrm U_x\,\mathrm O_k\,
  \mathrm U_y\,\mathrm O_j,
  \qquad \{x,y\}=\{i,j\}.
\]
The order \((x,y)=(i,j)\) again gives adjacent equal letters.  Hence
\((x,y)=(j,i)\), which is Equation~\eqref{eq:five-occurrence}. 
\end{proof}

\begin{proof}[proof of Theorem~\ref{thm:intro-source}]
Combining Lemma~\ref{lem:source-alternating} and
Lemma~\ref{lem:alternating-pattern} gives the classification of
Theorem~\ref{thm:intro-source}. Since there are exactly two alternating states, Lemma~\ref{lem:source-alternating} infers the upper bound of the number of source states.
\end{proof}

\subsection{Connecting the non-source states}\label{sec:non-source-scc}

We now prove that the source states are the only obstruction to strong
reachability.  Fix an alternating state \(\delta\), and put
\(X=\{0,\ldots,s-1\}\).  For \(S\subseteq X\), write
\[
  \delta_S=\delta+\sum_{i\in S}e_i,
  \qquad
  L_k=\{\delta_S:|S|=k\}.
\]
The integer \(k\) is the Hamming distance from \(\delta\).  By
the preceding observation, \(L_0=\{\delta\}\),
\(L_s=\{\delta+\one\}\), and \(L_1,\ldots,L_{s-1}\) are the
\emph{internal layers}.

We use the \emph{token digraph} of Fernandes et al.~\cite{FernandesEtAl2026}, the
directed analogue of the token graph of Fabila-Monroy et al.
\cite{FabilaMonroyEtAl2012}.  For a digraph \(D\) and  
\(1\le k<|V(D)|\) where \(V(D)\)  is the vertex set of \(D\), its token digraph \(T_k(D)\) has the \(k\)-subsets of
\(V(D)\) as vertices and an arrow
\[
  S\longrightarrow S\setminus\{x\}\cup\{y\}
\]
whenever \(x\to y\) is an arrow of \(D\), \(x\in S\), and \(y\notin S\).

\begin{definitionx}[Crossing digraph of an alternating state]
Since \(\delta\) is alternating, every adjacent pair in the cyclic word
\(W_\delta\) consists of \(\mathrm O_i\) and \(\mathrm U_j\) for unique
labels \(i,j\in X=\{0,\ldots,s-1\}\).  The \emph{crossing digraph} \(D_\delta\) has vertex
set \(X\) and an arrow \(i\to j\) whenever \(\mathrm O_i\) and
\(\mathrm U_j\) are cyclically adjacent in \(W_\delta\).
\end{definitionx}

The direction is from the \(\mathrm O\)-label to the \(\mathrm U\)-label.
The R1-reduced assumption excludes loops, and no ordered pair is joined by
parallel arrows.  If directions are forgotten while all edges are retained,
\(D_\delta\) is the abstract \(4\)-regular multigraph underlying the shadow
\(G\), whose vertices are the crossings and edges are the semi-arcs. 

For \(1\le k\le s-1\), let \(B_k^\delta\) be the subdigraph of
\(A_G[L_k]\) keeping exactly those arrows of \(A_G\) that come from an
arc crossing change on a semi-arc: the arrow \(\alpha\to\alpha'\) lies
in \(B_k^\delta\) precisely when some arc \(c\) of \(\alpha\) is a
semi-arc and the arc crossing change on \(c\) carries \(\alpha\) to
\(\alpha'\).

\begin{lemma}[Crossing and layer digraphs]\label{lem:strong-layers}
For \(1\le k\le s-1\), the map
\[
  \Phi_k:\{ S : S\subseteq X, |S|=k \} \longrightarrow L_k,
  \qquad S\longmapsto\delta_S,
\]
is a digraph isomorphism
\begin{equation}\label{eq:token-layer-isomorphism}
  T_k(D_\delta)\cong B_k^\delta.
\end{equation}
Moreover, \(D_\delta\) is strongly connected.  Consequently every internal
layer \(L_k\) is strongly connected in \(A_G\).
\end{lemma}

\begin{proof}
Fix \(k\).  The map \(\Phi_k\) is a bijection.  Consider a token arrow in $T_k(D_\delta)$
\[
  S\longrightarrow S'=S\setminus\{i\}\cup\{j\}, i\in S, j\notin S, 
\]
induced by \(i\to j\) in \(D_\delta\).  Thus \(\mathrm O_i\) and
\(\mathrm U_j\) are cyclically adjacent in \(W_\delta\).  In state \(\delta_S\), these two
occurrences are selected precisely when \(i\in S\) and \(j\notin S\).
\(\{O_i, U_j\}\) then form consecutive selections of \(\delta_S\).  The resulting  arc is a semi-arc and the arc crossing change gives an arrow in  \(B_k^\delta\) 
\[
  \Phi_k(S)=\delta_S\longrightarrow\delta_{S'}=\Phi_k(S').
\]
Conversely, the endpoints of a semi-arc crossponding to an arrow \(\delta_S\to \delta_{S'}\) in \(B_k^\delta\) are adjacent
occurrences \(\mathrm O_i,\mathrm U_j\).  Their selection by \(\delta_S\) again gives
\(i\in S\), \(j\notin S\), and the corresponding token arrow \(S\to S'\) in \(T_k(D_{\delta})\) replaces
\(i\) by \(j\).  This proves \eqref{eq:token-layer-isomorphism}.

It remains to prove that \(D_\delta\) is strongly connected.  Its underlying
undirected multigraph is the connected shadow \(G\).  Moreover, each vertex
satisfies
\begin{equation}\label{eq:crossing-digraph-degrees}
  \deg^+_{D_\delta}(i)=\deg^-_{D_\delta}(i)=2.
\end{equation}
For any
\(C\subseteq X\), each arrow internal to \(C\) is counted once as an
outgoing arrow and once as an incoming arrow.  Hence
\[
  \sum_{i\in C}\bigl(\deg^+(i)-\deg^-(i)\bigr)
  =|E(C,X\setminus C)|-|E(X\setminus C,C)|,
\]
where \(|E(A,B)|\) is the number of arrows with tail in \(A\) and head
in \(B\).
Now let \(C\) be a source SCC.  The second term is zero, while
Equation~\eqref{eq:crossing-digraph-degrees} makes the left side zero.
Therefore no arrow leaves \(C\).  Since the underlying undirected graph is
connected, \(C=X\). Hence \(D_\delta\) is strongly connected.

Finally, by \cite[Corollary~3.4]{FernandesEtAl2026} the token digraph
\(T_k(D_\delta)\) is strongly connected if and only if \(D_\delta\) is
strongly connected, which we have just shown.  Hence its isomorphic copy
\(B_k^\delta\), and therefore \(A_G[L_k]\), is strongly connected.
\end{proof}

\begin{lemma}[Bridges between layers]\label{lem:layer-bridges}
Suppose \(s>3\).  For \(1\le k\le s-3\), the layers \(L_k\) and
\(L_{k+2}\) lie in the same SCC of \(A_G\).
\end{lemma}

\begin{proof}
Choose three consecutive occurrences
\[
  \mathrm U_a\,\mathrm O_b\,\mathrm U_c
\]
in \(W_\delta\).  Their labels are pairwise distinct: adjacent labels are
distinct because \(W\) is R1-reduced, and \(a\ne c\) because a label has
only one \(\mathrm U\)-occurrence.  Choose a \(k\)-set \(S\) disjoint
from \(\{a,b,c\}\).  In \(\delta_S\), the occurrences
\(\mathrm U_a,\mathrm U_c\) are selected and bound an arc, so
\[
  \delta_S\longrightarrow\delta_{S\cup\{a,c\}}.
\]
This gives an arrow from \(L_k\) to \(L_{k+2}\).

For the reverse direction, choose consecutive occurrences
\[
  \mathrm O_a\,\mathrm U_b\,\mathrm O_c
\] in  \(W_\delta\). With the same argument, their labels are pairwise distinct.
Take a \((k-1)\)-set \(T\) disjoint from \(\{a,b,c\}\).  In
\(\delta_{T\cup\{a,b,c\}}\), the occurrences
\(\mathrm O_a,\mathrm O_c\) are selected and bound an arc.  Hence
\[
  \delta_{T\cup\{a,b,c\}}
  \longrightarrow
  \delta_{T\cup\{b\}},
\]
an arrow from \(L_{k+2}\) to \(L_k\).  Lemma~\ref{lem:strong-layers}
completes the proof.
\end{proof}

\begin{proof}[Proof of Theorem~\ref{thm:strong}]
Lemma~\ref{lem:layer-bridges} connect all internal layers whose indices have the same
parity.  Since
\[
  p(\delta_S)=p(\delta)+|S|\pmod2,
\]
all internal-layer states of a prescribed state parity lie in one SCC
\(C_\varepsilon\).
If an alternating state \(\alpha\in V_\varepsilon\) is not a source, choose
an incoming arrow \(\gamma\to\alpha\).  By
Lemma~\ref{lem:alternating-pattern}, \(\alpha\to\gamma\).  Relative to
\(\delta\), the state \(\gamma\) lies in \(L_2\) if \(\alpha=\delta\), and
in \(L_{s-2}\) if \(\alpha=\delta+\one\).  Both are internal because
\(s>3\).  Thus \(\alpha\in C_\varepsilon\).  This proves
\(N_\varepsilon=C_\varepsilon\).

A source $\alpha$ has no arrow
enters it.  It cannot share an SCC with any other vertex, and therefore a
singleton SCC.
Every successor \(x\) of $\alpha$ has positive indegree and hence is not a source, and
\(p(\alpha)=p(x)\).   
Since an arc
crossing change on an alternating state alters only two components and, for
\(s>3\), cannot produce a state differing from \(\alpha\) in all \(s\) components, i.e. $x\neq \alpha, \alpha+\one$. 
The source state $\alpha$  is alternating by
Lemma~\ref{lem:source-alternating}. Hence \(x\) is not alternating.
Just as above, it follows that \(x\in L_2\cup L_{s-2}\), and therefore
\(x\in N_\varepsilon\), where \(\varepsilon=p(\alpha)\).  Thus every source component is a singleton that
maps onto \(N_\varepsilon\), which gives the asserted condensation arrows.

\end{proof}

\subsection{Low-crossing shadows and corollaries}\label{sec:knot-interpretation}

\begin{remark}[Low-crossing shadows]\label{rem:low-crossings}
For \(s=1,2\), every classical shadow has an R1 curl and every resulting
diagram represents the unknot.  For \(s=3\), the classical low-crossing
classification gives a unique R1-reduced shadow, the trefoil shadow with
occurrence word \(012012\)~\cite{DowkerThistlethwaite1983}.  Its
two alternating states represent the right- and left-handed trefoils; the
other six states represent the unknot.  Directly from
Lemma~\ref{lem:alternating-pattern}, it has no source state and each parity class is
an SCC of size four.  The results below concern the standing case \(s>3\).
\end{remark}

\begin{corollary}[Exactly two weak components]\label{cor:weak}
The weak components of \(A_G\) are \(V_0\) and \(V_1\), each with
\(2^{s-1}\) states.
\end{corollary}

\begin{proof}
Parity separates the two sets.  Within \(V_\varepsilon\), the states in
\(N_\varepsilon\) are strongly connected, and every source state has an
arrow into \(N_\varepsilon\).
\end{proof}

\begin{corollary}[Size and parity]\label{cor:source-sink-parity}
If \(r\) is the number of source states, then \(r\le2\) and
\(\operatorname{SC}(A_G)\) has \(2+r\) vertices.  Its two sink vertices have sizes
\[
  |N_\varepsilon|=2^{s-1}-|S_\varepsilon|.
\]
If both sources exist, they have the same parity when \(s\) is even and
opposite parities when \(s\) is odd.
\end{corollary}

\begin{proof}
Only the parity assertion remains after Theorem~\ref{thm:strong}.  The two
possible sources are the complementary alternating states \(\alpha\) and
\(\alpha+\one\), whose parities differ by \(s\).
\end{proof}

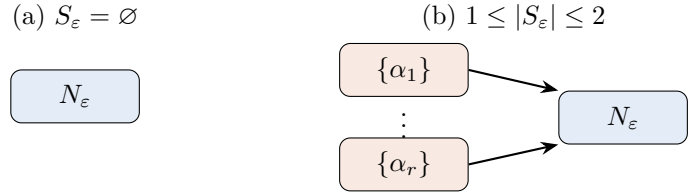
\begin{figure}[ht]
\centering
\begin{tikzpicture}[
  >=Stealth,
  component/.style={draw,rounded corners,minimum width=17mm,
    minimum height=7mm,inner sep=2pt},
  every node/.style={font=\small}
]
  \begin{scope}[shift={(-3.3,0)}]
    \node at (0,1.05) {(a) \(S_\varepsilon=\varnothing\)};
    \node[component,fill=RoyalBlue!9] (n0) at (0,0) {\(N_\varepsilon\)};
  \end{scope}
  \begin{scope}[shift={(2.5,0)}]
    \node at (0,1.05) {(b) \(1\le |S_\varepsilon|\le2\)};
    \node[component,fill=BrickRed!8] (s1) at (-1.45,0.35)
      {\(\{\alpha_1\}\)};
    \node at (-1.45,-0.28) {\(\vdots\)};
    \node[component,fill=BrickRed!8] (sr) at (-1.45,-0.9)
      {\(\{\alpha_r\}\)};
    \node[component,fill=RoyalBlue!9] (n1) at (1.45,-0.28)
      {\(N_\varepsilon\)};
    \draw[thick,-{Stealth}] (s1.east)--(n1.north west);
    \draw[thick,-{Stealth}] (sr.east)--(n1.south west);
  \end{scope}
\end{tikzpicture}
\caption{The two possible forms of one parity component of
\(\operatorname{SC}(A_G)\), where \(r=|S_\varepsilon|\).}
\label{fig:condensation-shape}
\end{figure}

\subsection{Examples with zero, one, and two sources}

The following examples realize all three possibilities allowed by
Corollary~\ref{cor:source-sink-parity}.

\begin{example}[Zero source states]\label{ex:no-source}
The connected sum of two trefoil shadows has occurrence word
\[
  W=012012345345.
\]
Its two alternating marked words contain, respectively, the forbidden blocks
\[
  \mathrm O_1\mathrm U_2\mathrm O_0\mathrm U_1\mathrm O_2,
  \qquad
  \mathrm O_0\mathrm U_1\mathrm O_2\mathrm U_0\mathrm O_1.
\]
Thus neither alternating state is a source.  The two parity classes are SCCs
of size \(32\), so \(\operatorname{SC}(A_G)\) consists of two isolated vertices.
\end{example}

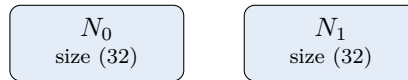
\begin{figure}[ht]
\centering
\begin{minipage}[c]{.52\textwidth}
  \centering
  \begin{tikzpicture}[
    sink/.style={draw,rounded corners,fill=RoyalBlue!9,
      minimum width=23mm,minimum height=10mm,align=center},
    every node/.style={font=\small}
  ]
    \node[sink] at (-1.55,0) {\(N_0\)\\[-2pt]{\scriptsize size (32)}};
    \node[sink] at (1.55,0) {\(N_1\)\\[-2pt]{\scriptsize size (32)}};
  \end{tikzpicture}
\end{minipage}
\caption{The condensation digraph  without
source}
\label{fig:no-source-condensation}
\end{figure}

\begin{example}[One source state]\label{ex:one-source}
Start with the trefoil word \(012012\).  Insert the figure-eight block
\(34564365\) between the second \(1\) and the final \(2\), and insert a
second figure-eight block \(7,8,9,10,8,7,10,9\) between that \(2\) and the
cyclic initial \(0\).  It gives the R1-reduced word
\[
\begin{aligned}
W={}&(0,1,2,0,1,3,4,5,6,4,3,6,5,2,\\
      &\qquad 7,8,9,10,8,7,10,9).
\end{aligned}
\]
Its alternating states are
\[
  \alpha=(0,1,0,1,0,1,0,0,1,0,1)
  \qquad\text{and}\qquad \alpha+\one.
\]
The marked word \(W_\alpha\) has no forbidden block from
Theorem~\ref{thm:intro-source}, whereas
\(W_{\alpha+\one}\) begins with
\[
  \mathrm O_0\,\mathrm U_1\,\mathrm O_2\,
  \mathrm U_0\,\mathrm O_1.
\]
Thus \(\alpha\) is the unique source.  Since \(p(\alpha)=1\), the three SCCs
have sizes \(1,1023,1024\), with \(\{\alpha\}\to N_1\) and \(N_0\)
isolated in the condensation digraph.
\end{example}

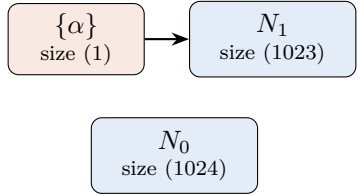
\begin{figure}[ht]
\centering
\begin{minipage}[c]{.39\textwidth}
  \centering
  \begin{tikzpicture}[
    >=Stealth,
    source/.style={draw,rounded corners,fill=BrickRed!8,
      minimum width=18mm,minimum height=9mm,align=center},
    sink/.style={draw,rounded corners,fill=RoyalBlue!9,
      minimum width=22mm,minimum height=10mm,align=center},
    every node/.style={font=\small}
  ]
    \node[source] (a) at (-1.25,.72) {\(\{\alpha\}\)\\[-2pt]
      {\scriptsize size (1)}};
    \node[sink] (n1) at (1.35,.72) {\(N_1\)\\[-2pt]
      {\scriptsize size (1023)}};
    \node[sink] (n0) at (.05,-.82) {\(N_0\)\\[-2pt]
      {\scriptsize size (1024)}};
    \draw[thick,-{Stealth}] (a.east)--(n1.west);
  \end{tikzpicture}
\end{minipage}
\caption{The condensation digraph with only one source}
\label{fig:one-source-condensation}
\end{figure}

\begin{example}[Two source states]\label{ex:figure-eight}
For the figure-eight shadow, one occurrence word is
\(W=01231032\), equivalent to the word in
\cite[Example~4.1 and Figure~10]{ChengLiaoSong2025}.  Its SCCs have sizes
\(1,1,6,8\).  The two singleton sources have the same parity and point to the
six-state sink; the other parity class is the eight-state SCC.
\end{example}

\begin{figure}[!ht]
\centering
\begin{minipage}[c]{.67\textwidth}
  \centering
  \begin{tikzpicture}[
    >=Stealth,
    source/.style={draw,rounded corners,fill=BrickRed!8,
      minimum width=18mm,minimum height=8mm,align=center},
    sink/.style={draw,rounded corners,fill=RoyalBlue!9,
      minimum width=21mm,minimum height=9mm,align=center},
    every node/.style={font=\small}
  ]
    \node[source] (a) at (0,.72) {\(\{\alpha\}\)\\[-2pt]
      {\scriptsize size (1)}};
    \node[source] (abar) at (0,-.72) {\(\{\alpha+\one\}\)\\[-2pt]
      {\scriptsize size (1)}};
    \node[sink] (ne) at (2.7,0) {\(N_0\)\\[-2pt]
      {\scriptsize size (6)}};
    \node[sink] (no) at (5.55,0) {\(N_1\)\\[-2pt]
      {\scriptsize size (8)}};
    \draw[thick,-{Stealth}] (a.east)--(ne.north west);
    \draw[thick,-{Stealth}] (abar.east)--(ne.south west);
  \end{tikzpicture}
\end{minipage}
\caption{The condensation digraph with two sources}
\label{fig:figure-eight-condensation}
\end{figure}
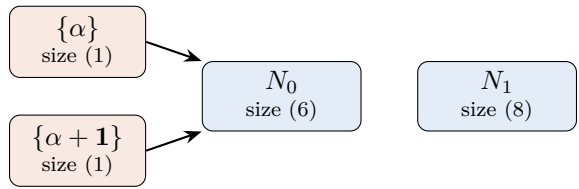

\FloatBarrier

\end{document}